\documentclass{amsart}

\usepackage{latexsym}
\usepackage{amsmath}
\usepackage{amssymb}
\usepackage{amsthm}
\usepackage{amscd}
\usepackage{color}

\newtheorem{theorem}{Theorem}[section]
\newtheorem{lemma}[theorem]{Lemma}

\theoremstyle{definition}

\newtheorem{definition}[theorem]{Definition}

\theoremstyle{remark}

\numberwithin{equation}{section}

\begin{document}

\title[Smoothing estimate on Morrey spaces]{Smoothing estimate for the heat semigroup with a homogeneous weight on Morrey spaces}

\author{Naoya Hatano and Masahiro Ikeda}

\address[Naoya Hatano]{Graduate School of Information Science and Technology, Osaka Uni-
versity, 1-5, Yamadaoka, Suita-shi, Osaka 565-0871, Japan,}

\address[Masahiro Ikeda]{Graduate School of Information Science and Technology, Osaka Uni-
versity, 1-5, Yamadaoka, Suita-shi, Osaka 565-0871, Japan / Center for Advanced
Intelligence Project RIKEN, 1-4-1, Nihonbashi, Chuo-ku, Tokyo 103-0027, Japan.}

\email[Naoya Hatano]{n.hatano.chuo@gmail.com}

\email[Masahiro Ikeda]{ikeda@ist.osaka-u.ac.jp / masahiro.ikeda@a.riken.jp}

\begin{abstract}
We study the smoothing estimate for the heat semigroup which is related to the nonlinear term of the Hardy-H\'enon parabolic equation on Morrey spaces.
This result is improvement of \cite[Proposition 3.3]{Tayachi20}, which is proved by using weak Lebesgue and Lorentz spaces.
\end{abstract}


\maketitle

\section{Introduction}

Looking ahead to analyzing the initial value problem
\begin{equation}\label{eq:HH}
\begin{cases}
\partial_tu-\Delta u=|x|^{-\gamma}|u|^{\alpha-1}u,
& (t,x)\in(0,\infty)\times{\mathbb R}^n, \\
u(0)=u_0,
\end{cases}
\tag{HH}
\end{equation}
where $u_0$ is a given function, and $\alpha>1$ and $\gamma>0$, we study the smoothing estimate for the heat semigroup with a homogeneous weight on Morrey spaces.
This equation is called Hardy-H\'enon parabolic equation, and the corresponding stationary problem was derived and studied by H\'enon \cite{Henon73}.
By Duhamel's principle, the integral equation can be given as follows:
\begin{equation}
u(t,x)
=
e^{t\Delta}u_0
+
\int_0^t
e^{(t-\tau)\Delta}[
|x|^{-\gamma}|u(\tau)|^{\alpha-1}u(\tau)
]
\,{\rm d}\tau,
\end{equation}
where $\{e^{t\Delta}\}_{t>0}$ is the heat semigroup which is defined by
\begin{equation}
e^{t\Delta}f(x)
\equiv
\frac1{\sqrt{(4\pi t)^n}}
\int_{{\mathbb R}^n}
\exp\left(-\frac{|x-y|^2}{4t}\right)
f(y)
\,{\rm d}y,
\quad
(t,x)\in(0,\infty)\times{\mathbb R}^n.
\end{equation}

In \cite{BTW17}, the local and global well-posedness in Lebesgue spaces are given.
Since the powered weight $|\cdot|^{-\gamma}$ belongs to the weak Lebesgue space $L^{n/\gamma,\infty}({\mathbb R}^n)$ and Lorentz space is generalization of the weak Lebesgue space, it is valid that one analyzes the equation \eqref{eq:HH} using the Lorentz spaces.
There are many works about well-posedness in the Lorentz spaces (see \cite{CITT24,Tayachi20} for example).

Ben Slimene, Tayachi and Weissler \cite{BTW17} provided the smoothing estimate for the heat semigroup as follows:
\begin{equation}\label{eq:Lp-Ls}
\|e^{t\Delta}[|\cdot|^{-\gamma}f]\|_{L^s}
\le C
t^{-\frac n2\left(\frac1p-\frac1s\right)-\frac\gamma2}
\|f\|_{L^p}
\end{equation}
when $n/s<\gamma+n/p<1$.
This estimate with the real interpolation implies the estimate
\begin{equation}\label{eq:wLp-wLs}
\|e^{t\Delta}[|\cdot|^{-\gamma}f]\|_{L^{s,\infty}}
\le C
t^{-\frac n2\left(\frac1p-\frac1s\right)-\frac\gamma2}
\|f\|_{L^{p,\infty}}
\end{equation}
appearing in \cite{Tayachi20} can be given by the real interpolation.
Moreover, Tayachi \cite{Tayachi20} provides the better smoothing estimate
\begin{equation}\label{eq:wLp-Ls}
\|e^{t\Delta}[|\cdot|^{-\gamma}f]\|_{L^s}
\le C
t^{-\frac n2\left(\frac1p-\frac1s\right)-\frac\gamma2}
\|f\|_{L^{p,\infty}}.
\end{equation}
In this paper, we give the improvement of the previous estimates by the proper embedding $L^{p,\infty}({\mathbb R}^n)\hookrightarrow{\mathcal M}^p_q({\mathbb R}^n)$.

\begin{definition}\label{def:wLp-Mpq}
Let $1\le q\le p<\infty$.
\begin{itemize}
\item[{\rm (1)}] The Lebesgue space $L^p({\mathbb R}^n)$ is defined as the space of all measurable functions $f$ with the finite norm
\[
\|f\|_{L^p}
\equiv
\left(
\int_{{\mathbb R}^n}|f(x)|^p\,{\rm d}x
\right)^{\frac1p}.
\]

\item[{\rm (2)}] The weak Lebesgue space $L^{p,\infty}({\mathbb R}^n)$ is defined as the space of all measurable functions $f$ with the finite quasi-norm
\[
\|f\|_{L^{p,\infty}}
\equiv
\sup_{t>0}t\|\chi_{\{x\in{\mathbb R}^n\,:\,|f(x)|>t\}}\|_{L^p}.
\]

\item[{\rm (3)}] The Morrey space ${\mathcal M}^p_q({\mathbb R}^n)$ is defined as the space of all measurable functions $f$ with the finite norm
\[
\|f\|_{{\mathcal M}^p_q}
\equiv
\sup_{x\in{\mathbb R}^n,t>0}
t^{\frac np-\frac nq}
\left(\int_{B(x,t)}|f(y)|^q\,{\rm d}y\right)^{\frac1q},
\]
where the symbol $B(x,t)$ is defined by the ball with radius $t>0$ and center $x\in{\mathbb R}^n$.
\end{itemize}
\end{definition}

The following is the main result.

\begin{theorem}\label{thm:Mpq-wLs}
Let $1<q\le p<\infty$, $1<s<\infty$ and $\gamma>0$, and assume that
\[
\frac nq-\frac np<\gamma<n-\frac np,
\]
Then the following assertions hold{\rm :}
\begin{itemize}
\item[{\rm (1)}] If $q\le s$, then there exists $C>0$ such that for all $t>0$ and $f\in{\mathcal M}^p_q({\mathbb R}^n)$,
\[
\|e^{t\Delta}[|\cdot|^{-\gamma}f]\|_{L^{s,\infty}}
\le C
t^{-\frac n2\left(\frac1p-\frac1s\right)-\frac\gamma2}
\|f\|_{{\mathcal M}^p_q}.
\]

\item[{\rm (2)}] If $q<s$, then there exists $C>0$ such that for all $t>0$ and $f\in{\mathcal M}^p_q({\mathbb R}^n)$,
\[
\|e^{t\Delta}[|\cdot|^{-\gamma}f]\|_{L^s}
\le C
t^{-\frac n2\left(\frac1p-\frac1s\right)-\frac\gamma2}
\|f\|_{{\mathcal M}^p_q}.
\]
\end{itemize}
\end{theorem}

We organize the remaining part of the paper as follows:
To prove the main theorem, we introduce the local Morrey-type spaces and its basic properties in Section \ref{s:pre}.
In Section \ref{s:Proof}, we prove the main theorem.
Finally, we give the remarks that the main theorem is improvement of the equation \eqref{eq:wLp-wLs}, and provide the extension of the main theorem in Section \ref{s:remark}.

\section{Preliminaries}\label{s:pre}

To prove the main result, we will use the local Morrey-type spaces.
Using these function spaces, we can return to the estimates for the Lebesgue spaces with a power weight.
After that we can use real interpolation, and be reduced to the estimate for the Morrey spaces.

\begin{definition}
Let $1\le q\le p<\infty$ and $1\le r\le\infty$.
The local Morrey-type space ${\rm L}{\mathcal M}^p_{q,r}({\mathbb R}^n)$  is defined as the space of all measurable functions $f$ with the finite norm
\[
\|f\|_{{\rm L}{\mathcal M}^p_{q,r}}
\equiv
\begin{cases}
\displaystyle
\left(
\int_0^\infty
\left[
t^{\frac np-\frac nq}
\left(
\int_{B(t)}|f(x)|^q\,{\rm d}x
\right)^{\frac1q}
\right]^r
\,\frac{{\rm d}t}t
\right)^{\frac1r},
& r<\infty, \\
\displaystyle
\sup_{t>0}
t^{\frac np-\frac nq}
\left(
\int_{B(t)}|f(x)|^q\,{\rm d}x
\right)^{\frac1q},
& r=\infty.
\end{cases}
\]
\end{definition}

By definition, we can see that
\begin{equation}\label{eq:M-LM}
\|f\|_{{\mathcal M}^p_q}
=
\sup_{x\in{\mathbb R}^n}
\|f(x-\cdot)\|_{{\rm L}{\mathcal M}^p_{q,\infty}}.
\end{equation}
Especially, the local Morrey-type space ${\rm L}{\mathcal M}^p_{q,\infty}({\mathbb R}^n)$ stands for the local Morrey space ${\rm L}{\mathcal M}^p_q({\mathbb R}^n)$.
Remark that when $s<\infty$, the local Morrey-type space ${\rm L}{\mathcal M}^p_{q,r}({\mathbb R}^n)$ must be nontrivial to satisfy $p<q$ (see \cite[Proposition 177]{SDH20}).
Then, here and below, we assume that $p<q$ when we use the local Morrey-type space ${\rm L}{\mathcal M}^p_{q,r}({\mathbb R}^n)$.

The embeddings for the local Morrey-type spaces are well known as follows.

\begin{lemma}[{\cite[Proposition 179 and Exercise 49]{SDH20}}]\label{prop:embedding-LM}
For the parameters $1\le q,q_1,q_2<p<\infty$ and $1\le r,r_1,r_2\le\infty$, the following assertions are hold{\rm :}
\begin{itemize}
\item[{\rm (1)}] If $q_1\ge q_2$, then
$
{\rm L}{\mathcal M}^p_{q_1,r}({\mathbb R}^n)
\hookrightarrow
{\rm L}{\mathcal M}^p_{q_2,r}({\mathbb R}^n)
$.

\item[{\rm (2)}] If $r_1\le r_2$, then
$
{\rm L}{\mathcal M}^p_{q,r_1}({\mathbb R}^n)
\hookrightarrow
{\rm L}{\mathcal M}^p_{q,r_2}({\mathbb R}^n)
$.
\end{itemize}
\end{lemma}

The special case of the norm of the local Morrey-type spaces can be written as the value at the origin of the Riesz potential given by Burenkov and Guliyev \cite[page 159]{BuGu04} as follows.
We gave the proof of the following lemma in Appendix \ref{App:LM-Ia} just in case.

\begin{lemma}\label{lem:LM-Ia}
Let $1<p<\infty$.
Then,
\[
\|f\|_{{\rm L}{\mathcal M}^p_{1,1}}
=
\left(n-\frac np\right)^{-1}
\int_{{\mathbb R}^n}\frac{|f(x)|}{|x|^{n-\frac np}}\,{\rm d}x.
\]
\end{lemma}

This lemma represents that the local Morrey-type space ${\rm L}{\mathcal M}^p_{q,q}({\mathbb R}^n)$ is a Lebesgue space with power weight.
By this chacterization, it is easy to see the following.

\begin{lemma}\label{prop:LMpqq-Ls}
Let $1\le q<p<\infty$, $1\le s\le\infty$ and $\gamma>0$.
If
\[
q\le s \quad
\gamma=\frac nq-\frac np,
\]
then, there exists $C>0$ such that for all $t>0$ and $f\in{\rm L}{\mathcal M}^p_{q,q}({\mathbb R}^n)$,
\[
\|e^{t\Delta}[|\cdot|^{-\gamma}f]\|_{L^s}
\le C
t^{-\frac n2\left(\frac1p-\frac1s\right)-\frac\gamma2}
\|f\|_{{\rm L}{\mathcal M}^p_{q,q}}.
\]
\end{lemma}

\begin{proof}
Using the $L^q$-$L^s$ smoothing estimate for the heat semigroup $\{e^{t\Delta}\}_{t>0}$ and Lemma \ref{lem:LM-Ia}, we have
\begin{align*}
\|e^{t\Delta}[|\cdot|^{-\gamma}f]\|_{L^s}
&\le C
t^{-\frac n2\left(\frac1q-\frac1s\right)}
\||\cdot|^{-\gamma}f\|_{L^q}\\
&= C
\left(n-\frac{nq}p\right)^{\frac1q}
t^{-\frac n2\left(\frac1p-\frac1s\right)-\frac\gamma2}
\|f\|_{{\rm L}{\mathcal M}^p_{q,q}},
\end{align*}
as desired.
\end{proof}

It is known that the real interpolation spaces for the local Morrey-type spaces can be calculated given by Burenkov and Nursultanov \cite{BuNu09}.
To obtain the main result, we may use the following statement to the previous lemma.

\begin{theorem}\label{thm:interpolation-LM}
Let $0<\theta<1$, $1\le q,p,p_0,p_1<\infty$ and $1\le r,r_0,r_1\le\infty$.
If
\[
\frac1p=\frac{1-\theta}{p_0}+\frac\theta{p_1},
\]
then the isomorphism
\[
(
{\rm L}{\mathcal M}^{p_0}_{q,r_0}({\mathbb R}^n),
{\rm L}{\mathcal M}^{p_1}_{q,r_1}({\mathbb R}^n)
)_{\theta,r}
\cong
{\rm L}{\mathcal M}^p_{q,r}({\mathbb R}^n)
\]
holds.
\end{theorem}

\section{Proof of Theorem \ref{thm:Mpq-wLs}}\label{s:Proof}

The case $p=q$ can be confirmed by the embedding $L^p({\mathbb R}^n)\hookrightarrow{\mathcal M}^p_q({\mathbb R}^n)$, immediately.
Thus we may assume that $p<q$.

\begin{itemize}
\item[(1)] Take $1<p_0,p_1<\infty$ and $1<s_0,s_1<\infty$ satisfying
\[
\frac nq-\frac n{p_0}
=
n-\frac n{p_1}
=
\gamma,
\quad
s_0>s>s_1,
\]
arbitrarily.
Then it is easy to see that $p_0,p_1$ satisfy $p_0>p>p_1$.
Using Lemma \ref{prop:LMpqq-Ls}, we have
\[
\|e^{t\Delta}[|\cdot|^{-\gamma}f]\|_{L^{s_0}}
\le C
t^{-\frac n2\left(\frac1{p_0}-\frac1{s_0}\right)-\frac\gamma2}
\|f\|_{{\rm L}{\mathcal M}^{p_0}_{q,q}}
\]
and
\[
\|e^{t\Delta}[|\cdot|^{-\gamma}f]\|_{L^{s_1}}
\le C
t^{-\frac n2\left(\frac1{p_1}-\frac1{s_1}\right)-\frac\gamma2}
\|f\|_{{\rm L}{\mathcal M}^{p_1}_{1,1}}
\le C
t^{-\frac n2\left(\frac1{p_1}-\frac1{s_1}\right)-\frac\gamma2}
\|f\|_{{\rm L}{\mathcal M}^{p_1}_{q,1}},
\]
where in the last inequality we have used Lemma \ref{prop:embedding-LM} (1).

Here, we assume that
\[
\frac{\dfrac1p-\dfrac1{p_0}}{\dfrac1{p_1}-\dfrac1{p_0}}
=
\frac{\dfrac1s-\dfrac1{s_0}}{\dfrac1{s_1}-\dfrac1{s_0}}.
\]
If we choose $\theta\in(0,1)$ by this quantity, it is easy to see that
\[
\frac1p=\frac{1-\theta}{p_0}+\frac\theta{p_1},
\quad
\frac1s=\frac{1-\theta}{s_0}+\frac\theta{s_1}.
\]
Then the real interpolation gives
\[
\|e^{t\Delta}[|\cdot|^{-\gamma}f]\|_{(L^{s_0},L^{s_1})_{\theta,\infty}}
\le C
t^{-\frac n2\left(\frac1p-\frac1s\right)-\frac\gamma2}
\|f\|_{
(
{\rm L}{\mathcal M}^{p_0}_{q,q},
{\rm L}{\mathcal M}^{p_1}_{q,1}
)_{\theta,\infty}
}.
\]
By Theorem \ref{thm:interpolation-LM}, we obtain
\[
\|e^{t\Delta}[|\cdot|^{-\gamma}f]\|_{L^{s,\infty}}
\le C
t^{-\frac n2\left(\frac1p-\frac1s\right)-\frac\gamma2}
\|f\|_{{\rm L}{\mathcal M}^p_{q,\infty}}
\le C
t^{-\frac n2\left(\frac1p-\frac1s\right)-\frac\gamma2}
\|f\|_{{\mathcal M}^p_q},
\]
where in the last inequality we used the equation \eqref{eq:M-LM}.
This is the desired result.

\item[(2)] Taking $s_0,s_1$ by $s_0>s>s_1\ge q$ instead of $s$ in (1), we have
\[
\|e^{t\Delta}[|\cdot|^{-\gamma}f]\|_{L^{s_0,\infty}}
\le C
t^{-\frac n2\left(\frac1p-\frac1{s_0}\right)-\frac\gamma2}
\|f\|_{{\mathcal M}^p_q}
\]
and
\[
\|e^{t\Delta}[|\cdot|^{-\gamma}f]\|_{L^{s_1,\infty}}
\le C
t^{-\frac n2\left(\frac1p-\frac1{s_1}\right)-\frac\gamma2}
\|f\|_{{\mathcal M}^p_q}.
\]
Hence, real interpolation gives
\[
\|e^{t\Delta}[|\cdot|^{-\gamma}f]\|_{L^s}
\le C
t^{-\frac n2\left(\frac1p-\frac1s\right)-\frac\gamma2}
\|f\|_{{\mathcal M}^p_q},
\]
as desired.
\end{itemize}

\section{Remarks}\label{s:remark}

It is known that $p>q$ implies the embedding
\[
L^{p,\infty}({\mathbb R}^n)
\hookrightarrow
{\mathcal M}^p_q({\mathbb R}^n)
\]
holds (see \cite[page 136]{Tsutsui10}).
Then we can see the estimate \eqref{eq:wLp-wLs} from Theorem \ref{thm:Mpq-wLs}.

Moreover, similar to the proof of Theorem \ref{thm:Mpq-wLs} (1), we have the extension of Theorem \ref{thm:Mpq-wLs} (1) by using (global) Morrey-type spaces.

\begin{definition}
Let $1\le p<\infty$ and $1\le q\le\infty$.
The Lorentz spaces $L^{p,q}({\mathbb R}^n)$ is defined as the space of all measurable functions $f$ with the finite norm
\[
\|f\|_{L^{p,q}}
\equiv
\begin{cases}
\displaystyle
\left(
\int_0^\infty
\left[
t\|\chi_{\{x\in{\mathbb R}^n\,:\,|f(x)|>t\}}\|_{L^p}
\right]^q
\,\frac{{\rm d}t}t
\right)^{\frac1q},
& q<\infty, \\
\displaystyle
\sup_{t>0}
t\|\chi_{\{x\in{\mathbb R}^n\,:\,|f(x)|>t\}}\|_{L^p},
& q=\infty.
\end{cases}
\]
\end{definition}

\begin{definition}
Let $1\le q\le p<\infty$ and $1\le r\le\infty$.
The {\rm (global)} Morrey-type space ${\mathcal M}^p_{q,r}({\mathbb R}^n)$  is defined as the space of all measurable functions $f$ with the finite norm
\[
\|f\|_{{\mathcal M}^p_{q,r}}
\equiv
\begin{cases}
\displaystyle
\sup_{x\in{\mathbb R}^n}
\left(
\int_0^\infty
\left[
t^{\frac np-\frac nq}
\left(
\int_{B(x,t)}|f(y)|^q\,{\rm d}y
\right)^{\frac1q}
\right]^r
\,\frac{{\rm d}t}t
\right)^{\frac1r},
& r<\infty, \\
\displaystyle
\sup_{x\in{\mathbb R}^n,t>0}
t^{\frac np-\frac nq}
\left(
\int_{B(x,t)}|f(y)|^q\,{\rm d}y
\right)^{\frac1q},
& r=\infty.
\end{cases}
\]
\end{definition}

\begin{theorem}\label{thm:Mpqr-Lsr}
Let $1<q<p<\infty$, $1\le r\le\infty$, $1<s<\infty$ and $\gamma>0$.
If
\[
q\le s, \quad
\frac nq-\frac np<\gamma<n-\frac np,
\]
then, there exists $C>0$ such that for all $t>0$ and $f\in{\mathcal M}^p_{q,r}({\mathbb R}^n)$,
\[
\|e^{t\Delta}[|\cdot|^{-\gamma}f]\|_{L^{s,r}}
\le C
t^{-\frac n2\left(\frac1p-\frac1s\right)-\frac\gamma2}
\|f\|_{{\mathcal M}^p_{q,r}}.
\]
\end{theorem}

Moreover, using the estimate
\[
\|\partial_t^k\partial_x^\alpha e^{t\Delta}f\|_{L^s}
\le C
t^{-\frac n2\left(\frac1q-\frac 1s\right)-k-\frac{|\alpha|}2}
\|f\|_{L^q}
\]
instead of the $L^q$-$L^s$ smoothing estimate for the heat semigroup $\{e^{t\Delta}\}_{t>0}$, we can generalize Theorems \ref{thm:Mpq-wLs} and \ref{thm:Mpqr-Lsr} as follows (see \cite[Subsection 1.1.3]{GGS}).

\begin{theorem}\label{thm:Mpqr-Lsr-2}
Let $k\in{\mathbb N}\cup\{0\}$, $\alpha\in({\mathbb N}\cup\{0\})^n$, $1<q\le p<\infty$, $1\le r\le\infty$, $1<s<\infty$ and $\gamma>0$, and assume that
\[
\frac nq-\frac np<\gamma<n-\frac np,
\]
Then the following assertions hold{\rm :}
\begin{itemize}
\item[{\rm (1)}] If $q\le s$, then there exists $C>0$ such that for all $t>0$ and $f\in{\mathcal M}^p_{q,r}({\mathbb R}^n)$,
\[
\|\partial_t^k\partial_x^\alpha e^{t\Delta}[|\cdot|^{-\gamma}f]\|_{L^{s,r}}
\le C
t^{
-\frac n2\left(\frac1p-\frac1s\right)
-k-\frac{|\alpha|}2-\frac\gamma2
}
\|f\|_{{\mathcal M}^p_{q,r}}.
\]

\item[{\rm (2)}] If $q<s$, then there exists $C>0$ such that for all $t>0$ and $f\in{\mathcal M}^p_q({\mathbb R}^n)$,
\[
\|\partial_t^k\partial_x^\alpha e^{t\Delta}[|\cdot|^{-\gamma}f]\|_{L^s}
\le C
t^{
-\frac n2\left(\frac1p-\frac1s\right)
-k-\frac{|\alpha|}2-\frac\gamma2
}
\|f\|_{{\mathcal M}^p_q}.
\]
\end{itemize}
\end{theorem}

This theorem is improvement of the estimate
\begin{equation}\label{eq:Lpr-Lsr}
\|
\partial_t^k\partial_x^\alpha e^{t\Delta}[|\cdot|^{-\gamma}f]
\|_{L^{s,r}}
\le C
t^{
-\frac n2\left(\frac1p-\frac1s\right)
-k-\frac{|\alpha|}2-\frac\gamma2
}
\|f\|_{L^{p,r}}
\end{equation}
given by the real interpolation of the estimate \eqref{eq:Lp-Ls}.
In fact, this claim is seen from the following theorem.

\begin{theorem}\label{thm:L-M}
Let $1\le q<p<\infty$ and $1\le r\le\infty$.
Then the embedding
\[
L^{p,r}({\mathbb R}^n)
\hookrightarrow
{\mathcal M}^p_{q,r}({\mathbb R}^n)
\]
holds, and this embedding is proper.
\end{theorem}

To prove this theorem, we use the following lemma.

\begin{lemma}\label{lem:L-M}
Let $1\le q<p<\infty$, and set
\[
g(x)\equiv
\chi_{B(1)}(x)+
\sum_{j=1}^\infty
(
\chi_{B(1)}(x-10^j{\bf e}_1)
+
\chi_{B(1)}(x+10^j{\bf e}_1)
),
\]
where ${\bf e}_1\equiv(1,0,\ldots,0)$.
Then
$
g\in{\mathcal M}^p_{q,1}({\mathbb R}^n)
\setminus
L^{p,\infty}({\mathbb R}^n)
$.
\end{lemma}

\begin{proof}
Since $g(x)$ is a indicator function over the union of pairwise disjoint infinity balls, then $g\notin L^{p,\infty}({\mathbb R}^n)$.
On the other hand, note that for each $j\in{\mathbb N}$,
\[
\int_{B(t)}|g(x)|^q\,{\rm d}x
\le C
\cdot
\begin{cases}
\min(1,t)^n, & 0<t\le10-1, \\
2j+1, & 10^j-1<t\le10^{j+1}-1.
\end{cases}
\]
Thus,
\begin{align*}
\|g\|_{{\mathcal M}^p_{q,1}}
&=
\|g\|_{{\rm L}{\mathcal M}^p_{q,1}}
=
\int_0^\infty
t^{\frac np-\frac nq}
\left(\int_{B(t)}|g(x)|^q\,{\rm d}x\right)^{\frac1q}
\,\frac{{\rm d}t}t\\
&\le C
\left(
\int_0^{10-1}
t^{\frac np-\frac nq}\min(1,t)^{\frac nq}
\,\frac{{\rm d}t}t
+
\sum_{j=1}^\infty
\int_{10^j-1}^{10^{j+1}-1}
t^{\frac np-\frac nq}\cdot(2j+1)^{\frac1q}
\,\frac{{\rm d}t}t
\right)\\
&= C
\left(
1+
\sum_{j=1}^\infty(2j+1)^{\frac1q}
(10^j-1)^{\frac np-\frac nq}
\right)\\
&\le C
\left(
1+
\sum_{j=1}^\infty(2j+1)^{\frac1q}
(10^j)^{\frac np-\frac nq}
\right)
<\infty,
\end{align*}
as desired.
\end{proof}

We give a proof of Theorem \ref{thm:L-M}.

\begin{proof}[Proof of Theorem {\rm \ref{thm:L-M}}]
Since
\[
\|f(x-\cdot)\|_{L^{p,r}}
=
\|f\|_{L^{p,r}},
\quad
\sup_{x\in{\mathbb R}^n}
\|f(x-\cdot)\|_{{\rm L}{\mathcal M}^p_{q,r}}
=
\|f\|_{{\mathcal M}^p_{q,r}},
\]
we may prove
$
L^{p,r}({\mathbb R}^n)
\hookrightarrow
{\rm L}{\mathcal M}^p_{q,r}({\mathbb R}^n)
$.
Take $p>1$ arbitrarily.
By H\"older's inequality for the Lorentz spaces,
\begin{align*}
\int_{{\mathbb R}^n}\frac{|f(x)|}{|x|^{n-\frac np}}\,{\rm d}x
\le C
\left\|
\frac1{|\cdot|^{n-\frac np}}
\right\|_{L^{p',\infty}}
\|f\|_{L^{p,1}}
= C
\|f\|_{L^{p,1}},
\end{align*}
and then
$
L^{p,1}({\mathbb R}^n)
\hookrightarrow
{\rm L}{\mathcal M}^p_{1,1}({\mathbb R}^n)
$
by Lemma \ref{lem:LM-Ia}.
Let $1<p_0<p_1<\infty$ and $0<\theta<1$ satisfy
\[
\frac1p=\frac{1-\theta}{p_0}+\frac\theta{p_1}.
\]
Then, using Theorem \ref{thm:interpolation-LM}, we obtain
\begin{align*}
L^{p,r}({\mathbb R}^n)
&\cong
(
L^{p_0,1}({\mathbb R}^n),L^{p_1,1}({\mathbb R}^n)
)_{\theta,r}\\
&\hookrightarrow
(
{\rm L}{\mathcal M}^{p_0}_{1,1}({\mathbb R}^n),
{\rm L}{\mathcal M}^{p_1}_{1,1}({\mathbb R}^n)
)_{\theta,r}\\
&\cong
{\rm L}{\mathcal M}^p_{1,r}({\mathbb R}^n).
\end{align*}

Next, we show that the embedding $L^{p,r}({\mathbb R}^n)\hookrightarrow{\mathcal M}^p_{q,r}({\mathbb R}^n)$ is proper.
By the embedding
$
L^{p,r}({\mathbb R}^n)
\hookrightarrow
L^{p,\infty}({\mathbb R}^n)
$
and Lemma \ref{prop:embedding-LM} (2), it suffices to show that
$
{\mathcal M}^p_{q,1}({\mathbb R}^n)
\setminus
L^{p,\infty}({\mathbb R}^n)
\ne\emptyset
$.
This claim is clear from Lemma \ref{lem:L-M}.

This is the desired result.
\end{proof}

\appendix
\section{Proof of Lemma \ref{lem:LM-Ia}}
\label{App:LM-Ia}

In this appendix, we give the proof of Lemma \ref{lem:LM-Ia}.
We estimate
\begin{align*}
\int_{{\mathbb R}^n}\frac{|f(x)|}{|x|^{n-\frac np}}\,{\rm d}x
&=
\left(n-\frac np\right)
\int_{{\mathbb R}^n}
\int_{|x|}^\infty\frac{|f(x)|}{t^{n-\frac np+1}}
\,{\rm d}t{\rm d}x\\
&=
\left(n-\frac np\right)
\int_0^\infty
t^{\frac np-n}
\int_{|x|<t}|f(x)|
\,\frac{{\rm d}x{\rm d}t}t\\
&=
\left(n-\frac np\right)
\|f\|_{{\rm L}{\mathcal M}^p_{1,1}},
\end{align*}
as desired.

{\bf Acknowledgements.} 
The authors are thankful to Professor Yoshihiro Sawano, Professor Toru Nogayama and Dr.\ Kazuki Kobayashi for their careful reading the paper and giving very helpful comments and advice.
The first-named author (N.H.) was supported by the Grant-in-Aid for Research Activity Start-up (No. 23K19013) and the Grant-in-Aid for JSPS Fellows (No. 25KJ0222).
The second author (M.I.) is supported by JSPS, the Grant-in-Aid for Scientific Research (C) (No. 25H01453) and the Grant-in-Aid for Transformative Research Areas (B) (No. 23K03174) and JST CREST, No. JPMJCR1913.

\bibliographystyle{amsplain}

\end{document}